\documentclass[12pt,twoside,a4paper,reqno]{amsart}

\usepackage{amsmath, amssymb, latexsym, amsthm, amsfonts, mathrsfs, amscd}
\usepackage{enumerate}[1.)]
\renewcommand{\eqref}[1]{(\ref{#1})}   %for some reason eqref is not supported properly
\usepackage{latexsym}
\usepackage{euscript}
\usepackage[margin=2.32 cm]{geometry}
\usepackage{epsfig}
\usepackage{hyperref}
\usepackage{tikz}
\usepackage{graphics}
\usepackage{array}
\usepackage{enumerate}
\usepackage[utf8]{inputenc}
\usepackage{xcolor}

\newtheorem{theorem}{Theorem}[section]
\newtheorem{lemma}[theorem]{Lemma}

\theoremstyle{definition}

\newtheorem{remark}{Remark}

\renewcommand{\l}{\left}
\renewcommand{\r}{\right}
\newtheorem{corollary}[theorem]{Corollary}

\newcommand{\modb}{(\textnormal{mod } b)}
\newcommand{\modq}{(\textnormal{mod }q)}
\newcommand{\modc}{(\textnormal{mod }c)}

\newcommand{\diff}{\textnormal{d}}
\newcommand{\Q}{{\mathbb Q}}
\newcommand{\R}{{\mathbb R}}
\newcommand{\Z}{{\mathbb Z}}
\newcommand{\C}{{\mathbb C}}
\newcommand{\N}{{\mathbb N}}
\newcommand{\abcdsum}{\sum_a\sum_b\sum_c\sum_d}

\newcommand{\be}{\begin{equation}}
\newcommand{\ee}{\end{equation}}
\newcommand{\ba}{\begin{equation}\begin{aligned}}
\newcommand{\ea}{\end{aligned}\end{equation}}
\numberwithin{equation}{section}
 2

\usepackage{setspace}

\newcommand{\ve}{\varepsilon}

\begin{document}
\title[Integral points]  {Counting $2\times2$ integer matrices with fixed trace and determinant}
\author{Rachita Guria}
\address {Department of Mathematics and Statistics\\
Indian Institute of Science Education and Research Kolkata}
\email{guriarachita2011@gmail.com}
\maketitle

\begin{abstract}
We count with a smooth weight the number of $2 \times 2$ integer matrices with a fixed characteristic polynomial with a main term and an error term using bounds for sums of Weyl sums for quadratic roots. 
\end{abstract}
	
%\blfootnote{2010 {\it Mathematics subject classification}: Primary 11M06; Secondary 11N36}
\section{Introduction}

For an affine variety $V$ given by integer polynomials $f_i \in \Z[x_1,\dots,x_n]:$

\be 
V = \l\{ x \in \C ^n : f_i(x)=0, i = 1,\dots, k\r \},\nonumber
\ee 
it is an interesting problem to investigate the asymptotic for the number of 
integral points on $V$ in an expanding region. To be precise, let
\be 
N(X,V) = \#\l\{ m \in V(\Z): \|m \| \leq X \r\}, \nonumber
\ee
where $V(A)$ denotes the set of points on $V$ defined over some subring $A$ of $\C$. Here $\|\cdot \|$ is a suitable norm on $\R^n$ .
Using  Harmonic Analysis on homogenous spaces, Duke, Rudnick and Sarnak \cite{DRS} established, under some natural technical conditions including  that $V(\R)$ is (affine) symmetric (see \cite[\S 1]{DRS}) and that 
the norm $\|\cdot \|$ is a rotation-invariant Euclidean norm,  an asymptotic formula for $N(X, V)$ as $X\to \infty$. In particular, if $V=V_{n, k}$, the set of $n\times n$ matrices  with determinant $k\neq 0$, their result gives the following asymptotic formula for the number of such matrices in a  norm-ball:
\[
 N(X, V_{n,k}) \sim c_{n,k}X^{n^2-n}
\]
as $X \to \infty$, where $c_{n,k}>0$ is an explicit constant. A related interesting problem is to count integer matrices whose characteristic polynomial is some fixed polynomial $P\in \Z[x]$. The method of \cite{DRS} 
is not directly applicable to this problem as the variety in this question is not (affine) symmetric. \\
Eskin, Mozes and Shah \cite{EMS} used techniques from Ergodic Theory to prove an asymptotic formula for more general varieties and in particular, if $V=V_{n, P}$, the set of $n\times n$ integer matrices  ($n\geq 2$) whose characteristic polynomials equal $P$, some fixed integer polynomial
of degree $n$ which is irreducible over $\Q$,  they showed
 that 
\be
N(X,V_{n, P}) \sim c_P X^{\frac{n(n-1)}{2}}, \quad X \rightarrow \infty,
\nonumber \ee 
for some constant $c_P>0$. Here,  the norm is given by $\|x\| = \sqrt{\sum_{ij}x_{ij}^2}$.
A simpler proof of this result was later given by Shah \cite{Shah} who also gave an explicit description of the constant $c_P$. 
The above result  due to Eskin, Mozes and Shah only gives an asymptotic relation and it is desirable to obtain an asymptotic formula with a precise main term and an error term. In other words, one would like to have a formula of the following kind:
\[
N(X, V_{n, P})=c_P X^{\frac{n(n-1)}{2}}+O(E(X)),
\]
as $X \to \infty$, where $E(X)=o(X^{\frac{n(n-1)}{2}})$. 

Using the theory of Weyl sums for roots of quadratic congruences due to Duke, Friedlander and Iwaniec \cite{DFI}, we  prove a ``smoothed" version of a result of this kind in the simplest case $n=2$ where we take $\|\cdot\|=\|\cdot\|_{\infty}$, i.e., for a matrix $A=(a_{ij})_{1\leq i,j\leq n}$, 
 we define $\|A\|=\max\{|a_{ij}|:1\leq i, j \leq n\}$. To state our result precisely, let $M_n(\Z)$ denote the set of $n \times n $ integer matrices for $n \geq 2 $.
Define, for fixed $t, r\in \Z$,
\be
S= \l\{\begin{pmatrix} a & b\\ c &
  d\end{pmatrix}\in M_2(\Z) : a+d =t, ad-bc =r \r\}.
  \ee
Hence, $S$ is the set of $2\times 2$ integral matrices with a fixed characteristic polynomial 
\be 
P(x) = x^2 - tx + r. 
  \ee
Then $N(X, V_{2, P})$ is the number of matrices $A\in S$ with $\|A\|_{\infty}\leq X$ and this is same as the sum
\be\label{sum}
S(X)=\mathop{\abcdsum}_{\substack
{a+d=t, ad-bc=r
  \\ 
  |a|, |b|,|c|,|d|\leq X}}
1.
\ee
  Now we consider a smoothed version of the above sum:
\be\label{smoothed}
S_w(X)= \mathop{\abcdsum}_{
a+d=t, ad-bc=r
  }
   w\l(\frac{a}{X}\r)w\l(\frac{b}{X}\r)w\l(\frac{c}{X}\r)w\l(\frac{d}{X}\r),
    \ee
where $w$ is a fixed smooth function with compact support in $\l[\frac{1}{2},1\r]$. Our Theorem is:

%%%%%%%%%%%%%%%%%%%%%%

\begin{theorem}\label{theorem} Let $t, r $ be integers with $t$ even  and let $D = \frac{t^2}{4} - r $ be a positive fundamental discriminant. Then, we have 
\be
S_w(X) = {\gamma}_{D}(1) M(X, D) + O(X^{1 - \frac{2}{1331}+\ve}),
\ee
where
\[
 M(X, D) = \iint \frac{1}{y} w\l(\frac{y}{X}\r) w\l(\frac{x}{X}\r)w\l(\frac{t-x}{X}\r)w\l(\frac{xt - x^2 - r }{yX}\r) \,\diff x \,\diff y,
 \]
and the constant $\gamma_D(1)$ is given in Lemma \ref{hooley}.
\end{theorem} 
%%%%%%%%%%%%%%%%%%%%%%%%%%%%
\begin{remark}
It is not hard to see that $M(X,D) \asymp X$ for $X<t<2X$, and $r = o(X^2)$ by the Mean Value Theorem of Calculus. 
\end{remark}

\begin{remark}
In this kind of counting problem, the choice of the norm plays a key role. In both \cite{DRS} and \cite{EMS}, it is necessary to take a rotation-invariant norm for the respective methods to work. On the contrary, for the method that we use, it is easiest to work with the Supremum norm. 
\end{remark}

\begin{remark}
It should be possible to prove a similar formula for the sum $S(X)$ given by \eqref{sum} but the proof  will be technically much more challenging. 
\end{remark}

\begin{remark}
Unlike \cite{EMS}, we do not need to assume that the polynomial $P$ is irreducible over $\Q$. 

\end{remark}
\begin{remark}
More generally, we can get a similar result for the number of integral points on an inhomogeneous quadric of the form $\alpha X^2 + \beta X + \gamma YZ + \delta = 0$, for $\alpha, \beta, \gamma, \delta \in \Z $.
\end{remark}

The following result, which can be compared with \cite[Theorem 2]{H}, is an immediate consequence of the proof of the theorem. 
\begin{corollary}
Let $g(u) = u^2 + 2pu +q $ be a monic integer polynomial such that $\Delta = p^2 - q$ is a positive fundamental discriminant. Then we have 
\be
 \sum_ n\tau _{X, w}(g(n)) w\l(\frac{n}{X}\r) = \gamma_{\Delta}(1) \iint \frac{1}{y} w\l(\frac{y}{X}\r)w\l(\frac{x}{X}\r)w\l(\frac{x^2 + 2px +q}{yX}\r)\,\diff x\,\diff y+ O(X^{1 - \frac{2}{1331}+\ve}) ,
 \ee
where $\tau_{X, w}$ is a smoothed restricted divisor function defined by:
\ba
\tau_ {X, w} (n) = \mathop{\sum_ b \sum _c}_{n = bc} w\l(\frac{b}{X}\r)w\l(\frac{c}{X}\r).
\nonumber \ea
\end{corollary}
\subsection*{Acknowledgements}
The author thanks Satadal Ganguly for helpful comments. This work is part of the author’s PhD thesis. She thanks
Indian Institute of Science Education and Research, Kolkata and all her teachers, especially Subrata Shyam Roy, for
all the support and also thanks Indian Statistical Institute, Kolkata where this work was carried out for an excellent
working atmosphere.
\section{Preliminaries}
In this section we collect together some known results that will be used in the proof of our theorem. 
\begin{lemma} Suppose that both $f$, $\hat{f}$ are in $L^1(\R)$ and have bounded variation. Then, for $q \in \N $, we have
\be\label{10}  \sum_{\substack{n \equiv \alpha \modq \\ n \in \Z}}f\left(n\right)= \frac{1}{q}\sum_{n \in \Z}e \left(\frac{\alpha n}{q}\right)\hat{f}\left(\frac{n}{q}\right) \ee
\end{lemma}
\begin{proof}
See e.g. \cite [Eq. (4.25)]{IK}.
\end{proof}
\begin{lemma}\label{hooley} Define \ba\rho(k)= \#\{\nu \modb :  \nu ^2  \equiv D (\textnormal{mod }k)\}.\ea
\\
For $y \geq 1$, we have \be \sum_{k \leq y} \rho(k) = y\gamma_{D}(1)+ O_{\ve}\l(y^{\frac{3}{4}}\r),\nonumber \ee 
where 
\be \gamma_{D}(s) = \frac{K(s)M(s)}{\zeta(2s)}.\nonumber\ee
Here
\be 
K(s)= (1+\frac{1}{2^s})^{-1}\sum\limits_{n=0}^{\infty}\frac{\rho(2^n)}{2^{ns}} \nonumber 
\ee 
and \be 
M(s)= \sum_{\substack{d^2|D \\ (d,2)=1}}\frac{d}{d^{2s}}\sum\limits_{\substack{l=1 \\ (l,2)=1}}^{\infty}\l(\frac{D/d^2}{l}\r)\frac{1}{l^s} = \sum_{\substack{d^2|D \\ (d,2)=1}}\frac{d}{d^{2s}} L_{\l(\frac{D}{d^2}\r)}(s).  \nonumber 
\ee
\end{lemma}
\begin{proof}
See \cite[Eq. (7)]{H}. The $O$-term is independent of $D$.
\end{proof}

\section{Proof of Theorem \ref{theorem} }

Let us briefly outline the method of the proof. At first, we replace $d$ by $t-a$, and remove $c$ in the resulting equation $a(t-a)-bc=r$ by the congruence condition $ at - a^2 - r \equiv 0 (\textnormal{mod }b) $. At this stage, a trivial estimation gives $S_w(X)\ll X$. Next, we use the Poisson summation formula to evaluate the sum over $a$. We separate the zero frequency and the non-zero frequencies. The zero frequency yields the main term, whereas in regard to the non-zero frequencies, the key observation  is that the sum over $b$ is precisely a sum of the Weyl sums for quadratic roots averaged over the moduli $b \sim X $ weighted by a nice test function and we use \cite[ Theorem 1.1]{DFI} to estimate it.
\begin{proof}
\ba
S_w(X)&= \mathop{\sum_a \sum_b \sum_c}_{at - a^2 - r = bc}w\l(\frac{a}{X}\r)w\l(\frac{t-a}{X}\r)w\l(\frac{b}{X}\r)w\l(\frac{c}{X}\r)\\
&= \mathop{\sum_a \sum_b }_{at - a^2 - r \equiv 0 \modb }w\l(\frac{a}{X}\r)w\l(\frac{t-a}{X}\r)w\l(\frac{at - a^2 - r }{bX}\r)w\l(\frac{b}{X}\r)\\
& = \mathop{\sum_a \sum_b }w\l(\frac{a}{X}\r)w\l(\frac{t-a}{X}\r)w\l(\frac{at - a^2 - r }{bX}\r)w\l(\frac{b}{X}\r) \frac{1}{b}\sum_{h \modb} e\l(\frac{h(at - a^2 - r)}{b}\r).\nonumber
 \ea We split the inner sum into residue classes modulo $b$ and apply Possion summation formula on the $a$-sum \ba \label{a_1 sum}
\sum_ a w\l(\frac{a}{X}\r)w\l(\frac{t-a}{X}\r)w\l(\frac{at - a^2 - r }{bX}\r)e\l(\frac{h(at - a^2)}{b}\r) &= \frac{1}{b}\sum_{n \in \Z} \hat{f}\l(\frac{n}{b}\r)\\
&\sum _{\nu \modb} e\l(\frac{h(\nu t- \nu ^2)+n\nu}{b}\r),
\ea where \be f(x)= w\l(\frac{x}{X}\r)w\l(\frac{t-x}{X}\r)w\l(\frac{xt - x^2 - r }{bX}\r). \nonumber \ee 
Note that, by repeated integration by parts, one can show that $\hat{f}\l(\frac{n}{b}\r)$ is negligible unless $n \ll X^{\ve}$ for any $\ve >0.$
Substituting \eqref{a_1 sum}, we can recast $S_w(X)$ as \ba
S_w(X)&= \sum_{n \in \Z}  \sum_{b} \frac{1}{b^2}w\l(\frac{b}{X}\r)\hat{f}\l(\frac{n}{b}\r)\sum _{\nu \modb} e\l(\frac{n\nu}{b}\r)\sum_{h \modb} e\l(\frac{h(\nu t - \nu ^2 - r)}{b}\r)\\
&= \sum_{n \in \Z}  \sum_{b} \frac{1}{b}w\l(\frac{b}{X}\r)\hat{f}\l(\frac{n}{b}\r)\sum _{\substack{\nu \modb \\ \nu ^2 -\nu t+ r \equiv 0 \modb }} e\l(\frac{n\nu}{b}\r).
\nonumber \ea Now we will treat the zero frequency and the non-zero frequency of the $n$-sum separately and thus write
\ba S_w(X) = A(X)+B(X), \ea where

\be A(X) = \int w\l(\frac{x}{X}\r)w\l(\frac{t-x}{X}\r)\l\{ 
\sum_{b} \frac{1}{b}w\l(\frac{b}{X}\r)w\l(\frac{xt - x^2 - r }{bX}\r) \sum _{\substack{\nu \modb \\  \nu ^2- \nu t  + r \equiv 0 \modb }} 1 \r\}\diff x \ee

and

\ba \label{b_sum} B(X)& = \int w\l(\frac{x}{X}\r)w\l(\frac{t-x}{X}\r)\l\{ \sum_{\substack{n \in \Z \\ n \neq 0}} \sum_{b} \frac{1}{b}w\l(\frac{b}{X}\r)w\l(\frac{xt - x^2 - r }{bX}\r)e \l(\frac{-nx}{b}\r)\r.\\
&\l.\sum _{\substack{\nu \modb \\ \nu ^2 -\nu t + r \equiv 0 \modb }}e\l(\frac{n\nu}{b}\r)\r\} \diff x.
\ea

\subsection{The Main Term} In our case, \ba\rho(b)= \#\{\nu \modb :  \nu ^2 - \nu t + r \equiv 0 \modb\}. \nonumber\ea 
Clearly,
\ba 
A(X) = \int w\l(\frac{x}{X}\r)w\l(\frac{t-x}{X}\r)\l\{ \sum_{b} \frac{1}{b}w\l(\frac{b}{X}\r)w\l(\frac{xt - x^2 - r }{bX}\r)\rho(b) \r\}\diff x .\nonumber 
\ea 

By partial summation on the $b$-sum, we obtain
\ba
A(X) = -\int  w\l(\frac{x}{X}\r)w\l(\frac{t-x}{X}\r) \l[\int \frac{d}{dy}\l(\frac{1}{y} w\l(\frac{y}{X}\r)w\l(\frac{xt - x^2 - r }{yX}\r)\r) \l(\sum_{b \leq y} \rho(b)\r)  \,\diff y\r] \,\diff x .\nonumber
 \ea
Applying Lemma \ref{hooley} to the $b$-sum, and using integration by parts later on the $y$ integral,
\ba\label{A} A(X) &= \gamma_{\frac{t^2}{4}-r}(1)\iint \frac{1}{y} w\l(\frac{y}{X}\r) w\l(\frac{x}{X}\r)w\l(\frac{t-x}{X}\r)w\l(\frac{xt - x^2 - r }{yX}\r) \,\diff x \,\diff y \\
&+ O_{\ve}(X^{\frac{3}{4}}).\ea
\subsection{The Error Term} We start the estimation by stating the following theorem from \cite[Theorem 1.1]{DFI}. \begin{theorem}\label{thm_1}
Let $h \geq 1, q\geq 1$, and D a positive fundamental discriminant. Let $v$ be a smooth function supported on $Y \leq y \leq 2Y $ with $Y \geq 1$, such that 
\ba \l|v(y)\r|\leq 1, \quad \quad \quad Y^3\l| v'''(y)\r| \leq 1.\nonumber \ea Define the Weyl sum over roots of quadratic congruences: \ba W_h(D;c) = \sum_{\substack{b \modc \\ b^2 \equiv D \modc}}e\l(\frac{hb}{c}\r), \nonumber\ea 
and \be W_h(D) = \sum_{c \equiv 0 (\textnormal{mod }q)}v(c) W_h(D;c). \ee  
Then \ba W_h(D) \ll h^{\frac{1}{4}}\l(Y+ h\sqrt{D}\r)^{\frac{3}{4}}D^{\frac{1}{8}-\frac{1}{1331}},\nonumber \ea where the implied constant is absolute.
\end{theorem}  Rearranging \eqref{b_sum}, we write $B(X)$ as \ba B(X)& = \int w\l(\frac{x}{X}\r)w\l(\frac{t-x}{X}\r)\l\{\sum_{\substack{n \in \Z \\ n \neq 0}} \sum_{b} \frac{1}{b}w\l(\frac{b}{X}\r)w\l(\frac{xt - x^2 - r }{bX}\r)e \l(\frac{-nx}{b}\r)e\l(\frac{nt}{2b}\r)\r.\\
&\l. \sum _{\substack{\nu  \modb \\ \nu  ^2 \equiv \frac{t^2}{4}-r \modb }}e\l(\frac{n\nu }{b}\r) \r\}\diff x .\nonumber\ea
Now we are ready to apply Theorem \ref{thm_1}. Let us take  \ba v(y)=\frac{X}{y}w\l(\frac{y}{X}\r)w\l(\frac{xt - x^2 - r }{yX}\r)e \l(\frac{-nx}{y}\r)e\l(\frac{nt}{2y}\r),\nonumber \ea which satisfies all the conditions of the theorem with $D = \frac{t^2}{4}-r$, and $h =n$. 
\ba \label{B} B(X) &\ll \frac{1}{X}\int w\l(\frac{x}{X}\r)w\l(\frac{t-x}{X}\r) X^{1 - \frac{2}{1331}+\ve}\diff x\\
& \ll X^{1 - \frac{2}{1331}+\ve}.   \ea

By \eqref{A} and \eqref{B}, we can therefore conclude that

\ba S_w(X) &= \gamma_{\frac{t^2}{4}-r}(1)\iint \frac{1}{y} w\l(\frac{y}{X}\r) w\l(\frac{x}{X}\r)w\l(\frac{t-x}{X}\r)w\l(\frac{xt - x^2 - r }{yX}\r) \,\diff x \,\diff y \\
&+ O\l(X^{1 - \frac{2}{1331}+\ve}\r). \nonumber \ea
\end{proof}

\end{document}